\documentclass{article}
\usepackage{amsmath,amssymb,color}
\usepackage{hyperref}
\usepackage{graphicx}
\usepackage{float}
\usepackage{caption, subcaption}
\usepackage[framemethod=TikZ]{mdframed}

\usepackage{fullpage}
\usepackage{cite}
\numberwithin{equation}{section}

\graphicspath{{./graphics/}}

\mdfdefinestyle{MyFrame}{%
    linecolor=orange,
    outerlinewidth=15pt,
    roundcorner=10pt,
    innertopmargin=2\baselineskip,
    innerbottommargin=2\baselineskip,
    innerrightmargin=20pt,
    innerleftmargin=20pt,
    backgroundcolor=blue!70!white}

\newcommand{\R}{\mathbb{R}}

\newtheorem{thm}{Theorem}[section]
\newtheorem{lem}[thm]{Lemma}

\newtheorem{defn}[thm]{Definition}
\newtheorem{example}[thm]{Example}
\newtheorem{remark}[thm]{Remark}

\newenvironment{proof}{{\bf Proof.}  }{\hfill$\blacksquare$}

\newcommand{\0}{\mathbf{0}}

\newcommand{\av}[1]{\left|{#1}\right|}

\newcommand{\norm}[1]{\left\|{#1}\right\|}

\newcommand{\be}{\begin{enumerate}}
\newcommand{\bi}{\begin{itemize}}
\newcommand{\ee}{\end{enumerate}}
\newcommand{\ei}{\end{itemize}}
\newcommand{\ii}{\item}

\newcommand{\p}{{\mathbf {{p}}}}

\newcommand{\w}{{\mathbf{w}}} 
\newcommand{\x}{{\mathbf {{x}}}}
\newcommand{\y}{{\mathbf {{y}}}}

\renewcommand{\L}{\mathcal{L}}

\renewcommand{\L}{\mathcal{L}}

\renewcommand{\phi}{\varphi}

\begin{document}
\title{Stable Configurations in Social Networks} \date{\today}

\author{Jared C. Bronski \\ University of Illinois \\ \and Lee DeVille \\ University of Illinois
 \\ \and Tim Ferguson \\ University of Illinois \\ \and Michael Livesay \\ University of Illinois \\ }

\maketitle

\begin{abstract}
We present and analyze a model of opinion formation on an arbitrary network whose dynamics comes from a global energy function. We study the global and local minimizers of this energy, which we call stable opinion configurations, and describe the global minimizers under certain assumptions on the friendship graph.  We show a surprising result that the number of stable configurations is not necessarily monotone in the strength of connection in the social network, i.e. the model sometimes supports more stable configurations when the interpersonal connections are made stronger. 
\end{abstract}


\section{Introduction} \label{sec:intro}

\subsection{Social Network Models}

Over the last couple of decades, there has been a large degree of interest in models of general dynamical systems defined on networks~\cite{Strogatz.01,  Newman.Barabasi.Watts.book,  Newman.Girvan.04, Jin.Girvan.Newman.01} and in particular models of social or biological dynamics on networks~\cite{Watts.Strogatz.98, Newman.Girvan.04, Castellano.Fortunato.Loreto.09}.  Of course a wide variety of models and potential applications exist, but one simple context is the study of opinion formation on a network.  The classical voter model~\cite{Holley.Liggett.75, Holley.Liggett.78, Durrett.Levin.94, Liggett.book} and its generalizations~\cite{Sood.Antal.Redner.08} were one of the first models considered for opinion formation, but a variety of more complicated models exist~\cite{Sznajd.00, Kunegis.Lommatzsch.Bauchhage.09, Yildiz.etal.11, Altafini.12, Kunegis.etal.12, Ghaderi.Srikant.13, Auoay.etal.14, DaiPra.Louis.Minelli.14, Das.Gollapudi.Mungala.14, Berghardt.Rand.Girvan.16}.

In this paper, we consider a relatively simple dynamical model of social opinion formation whose dynamics are given by a single global potential function.  In our model, there are $n$ agents, each of which holds an opinion represented by a scalar quantity.  In the absence of any interaction with the other individuals, each agent will relax to one of two opinions, each of which is the negative of each other --- more specifically, each individual relaxes in a symmetric double well potential.  We then add on a coupling between all of the agents, but we allow for the coupling terms to be both positive and negative.   Thus we allow for both ``friends'' and ``enemies'' in this network, with the idea that one's opinions move towards those of one's friends, and away from those of one's enemies. 

In our model, all of the dynamics can be represented as a gradient flow in a potential; therefore, while our model might have multiple stable configurations, we can compare them energetically and determine which is the ``most stable'' configuration, i.e. the one which globally minimizes the potential.  In this sense, our model is quite reminiscent of both~\cite{Ashwin.etal.17} and~\cite{Berglund.Fernandez.Gentz.07.1, Berglund.Fernandez.Gentz.07.2}; in fact the latter two papers are a study of our model with only friendly connections in certain graph topologies. 

The main results of this paper are twofold.  In Section~\ref{sec:global} we describe the global mimimizers of the energy functional whenever the graph topology is ``balanced'' --- a condition on the graph which can best be summarized as ``the enemy of my enemy is my friend''.  In Section~\ref{sec:nonmonotone} we show that the dependence of the global system on the strength of the coupling can be quite complicated, and we show that increasing the coupling strength can both increase and decrease the number of minima.

\subsection{Description of Model}

\begin{defn}
Let $G = (V,E,\Gamma)$ be an undirected weighted graph; here $V = \{1,2,\dots,n\}$ are the vertices of the graph, $E  \subseteq V \times V$ are the edges, and $\gamma_{ij} = \gamma_{ji}$ is the weight on the edge joining the vertices $i, j$.  When necessary, we will denote by $\mathcal{G}_n$ as the set of all such graphs with $n$ vertices.
\end{defn}

\begin{defn}
Given a graph $G \in \mathcal{G}_n$, a function $W\colon\R \to \R$, and a parameter $\kappa \in [0,\infty)$ we define the energy
\begin{equation}\label{eq:defofE} 
E_{W,G,\kappa}(\x) := \sum_{i=1}^n W(x_i) + \frac\kappa2 \sum_{i,j=1}^n \gamma_{ij}(x_i-x_j)^2.
\end{equation}
The model we consider is the gradient flow for this energy, namely
\begin{equation}\label{eq:ode}
\frac{dx_i}{dt} = - \frac{\partial}{\partial x_i} E_{W,G,\kappa}(\x) = \kappa (\L(G)\x)_i - W'(x_i),
\end{equation}
where $\L(G)$ is the graph Laplacian whose components are given by 
\begin{equation}\label{eq:defofL}
  (\L(G))_{ij} = \begin{cases} \gamma_{ij}, & i\neq j,\\ -\sum_{k\neq i} \gamma_{ik}, &i=j.\end{cases}
\end{equation}
\end{defn}

Since we consider the gradient flow, the attracting fixed points of~\eqref{eq:ode} are exactly the local minima of~\eqref{eq:defofE}.  These minima are the main object of study in the current paper.

The functions $W$ correspond to the dynamics of an individual's opinion when uncoupled from others, and the $\gamma_{ij}$ encode the strength of interaction between individuals.  By setting $\kappa=0$, we can turn off the interaction, and we see that each individual independently moves to a local minimum of the function $W$.  Setting $\kappa$ large makes the interactions between individuals dominate.  We stress here that we do not assume $\gamma_{ij}\ge 0$, which would lead to only a friendly attracting force between individuals --- we let $\gamma_{ij}<0$, so that individuals who are enemies will have opinions that repel.   We will colloquially refer to $W$ as the ``individual potential'' and the $\gamma_{ij}$ as the ``interaction strengths''.

Now we define the set $\mathcal{W}$ of allowable potentials:

\begin{defn} \label{defofW}
Let $\mathcal{W}$ be the set of functions $W \colon \R \rightarrow \R$ such that $W$ are in $C^2(\R)$ and even, that there is some $m$ with $W'(\pm m) = 0$, $W'(x) > 0$ if $x \in (-m,0) \cup (m,\infty)$ and $W'(x) < 0$ if $x \in (-\infty,-m) \cup (0,m)$, and finally 
\begin{equation*}
  \lim_{x\to\pm\infty}\frac{W(x)}{x^2} = \infty.
\end{equation*}
\end{defn}
One can easily check that $W(x) = \frac{1}{4}(1-x^2)^2$ is in the class $\mathcal{W}$  with $m=1$, and we will refer to this in some cases as the ``classical potential''.  Since the interaction energy grows quadratically, we choose the potential $W$ to be coercive enough so that the set of minima of~\eqref{eq:defofE} are bounded; in fact, one can obtain bounds on their locations by knowing $W$ and the magnitudes of the $\gamma_{ij}$ (we discuss this further below). 

It is not hard to see that if we choose $\kappa= 0$, then there are $2^n$ minima, each with energy $0$, at the points $(\pm m,\pm m,\dots, \pm m)$; in this case these are all also global minima.  As we increase $\kappa$, we can expect several things to occur:  for some range of $\kappa$ near zero, minima might move but will persist, but of course only some of them will remain global minima.  We also expect that minima can disappear under bifurcations (and in fact it is shown in~\cite{Berglund.Fernandez.Gentz.07.1} that all bifurcations cause minima to disappear under certain conditions.

\begin{example}

Consider the graph on three vertices with edge weights 1, 1, and $-2$, where $W(x) = \frac{1}{4} (1-x^2)^2$, so that our entire energy function is
\begin{equation*}
  \sum_{i=1}^3 \frac14(1-x_i^2)^2 + \frac\kappa 2(x_1-x_2)^2 + \frac\kappa 2(x_1-x_3)^2  - \kappa(x_2-x_3)^2.
\end{equation*}
In Figures~\ref{fig:num_mins},~\ref{fig:3D} below we give two plots showing how the minima evolve as $\kappa$ increases from zero to infinity.

\begin{figure}
\centering
\begin{subfigure}{.5\textwidth}
  \centering
  \includegraphics[width=1\linewidth]{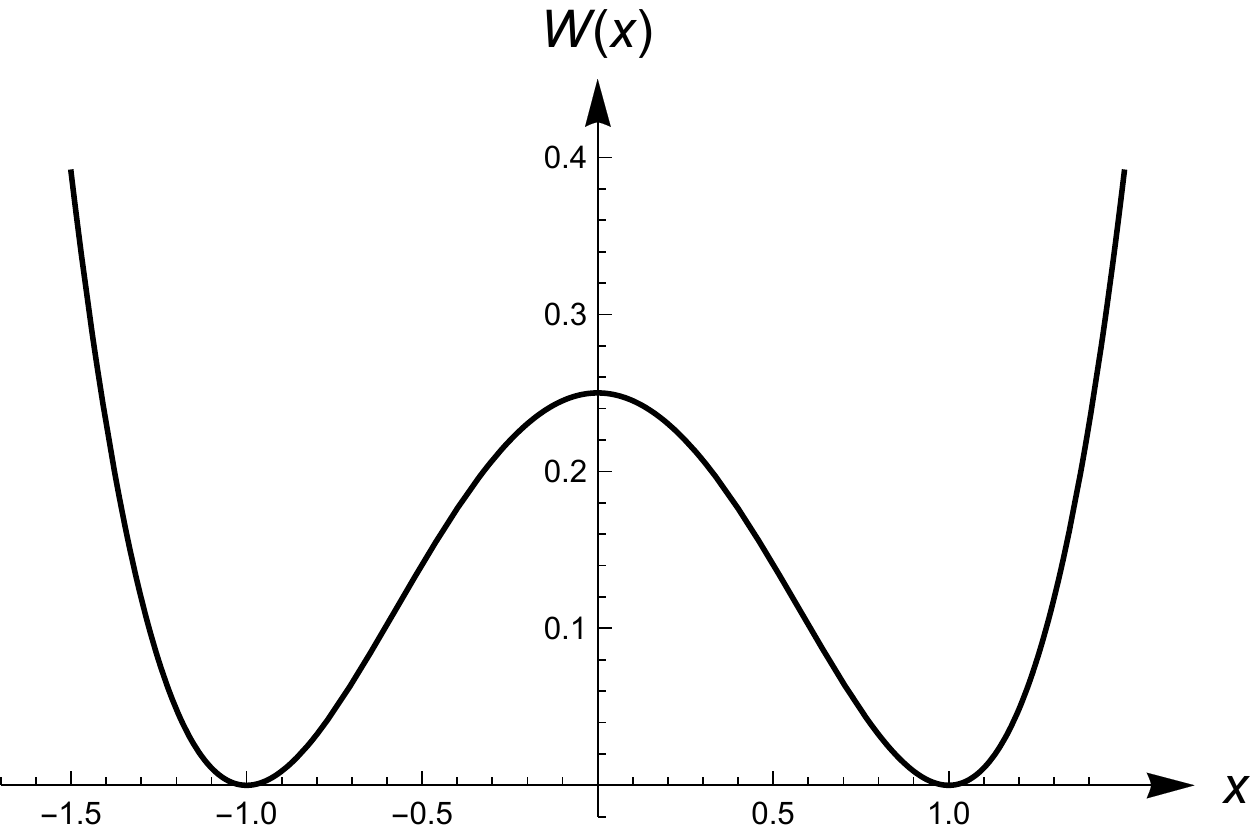}
\end{subfigure}%
\begin{subfigure}{.5\textwidth}
  \centering
  \includegraphics[width=1\linewidth]{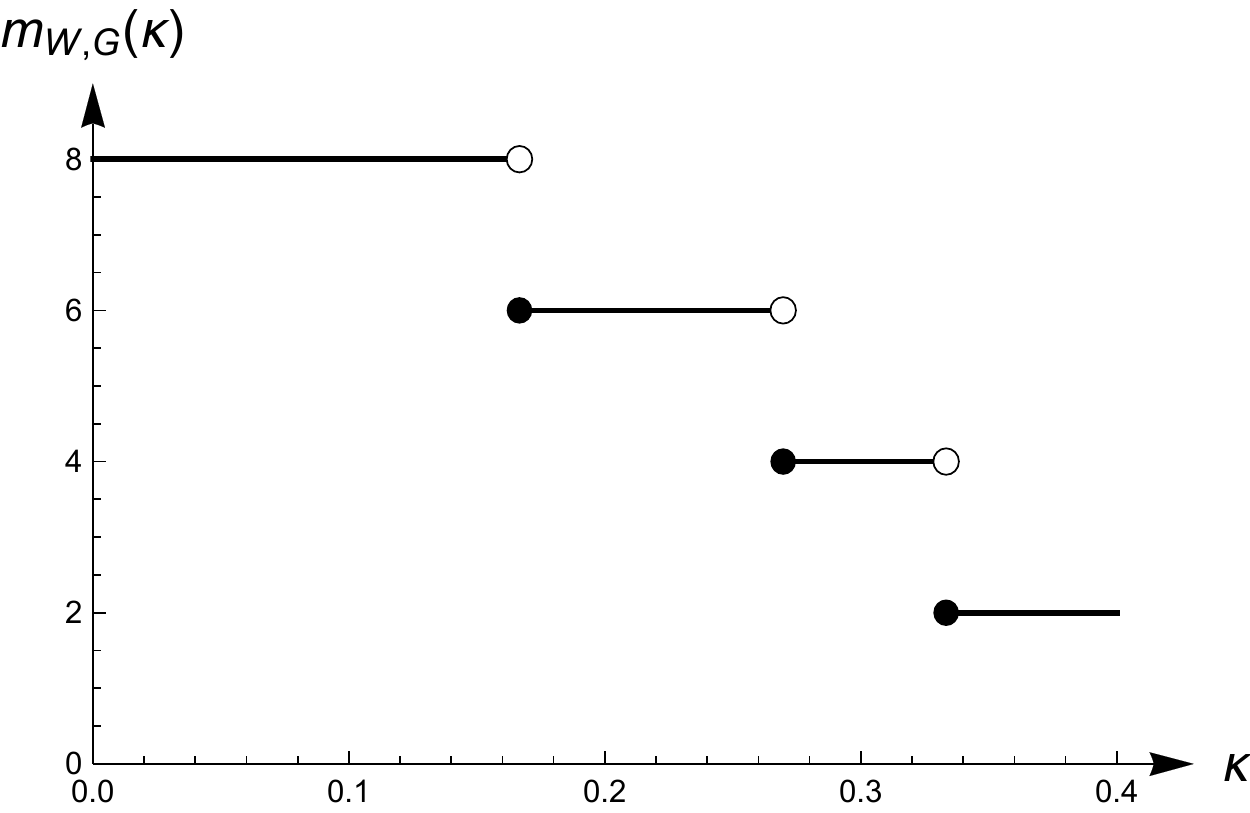}
\end{subfigure}
\caption{Plots of the potential $W$ as well as the number number of minima $m_{W,G}$. We can explicitly compute the points of discontinuity to be $\kappa = 1/6, \phi/6, 1/3$ with $\x = (0,0,0)$, $\x = \pm(0,\phi,-\phi)$, and $\x = \pm(1,1,1)$ respectively.  Here we use $\phi$ for the golden ratio, i.e. the larger root of $x^2-x-1$.}
\label{fig:num_mins}
\end{figure}

\begin{figure}[H]
\center
\includegraphics[width=.7\textwidth]{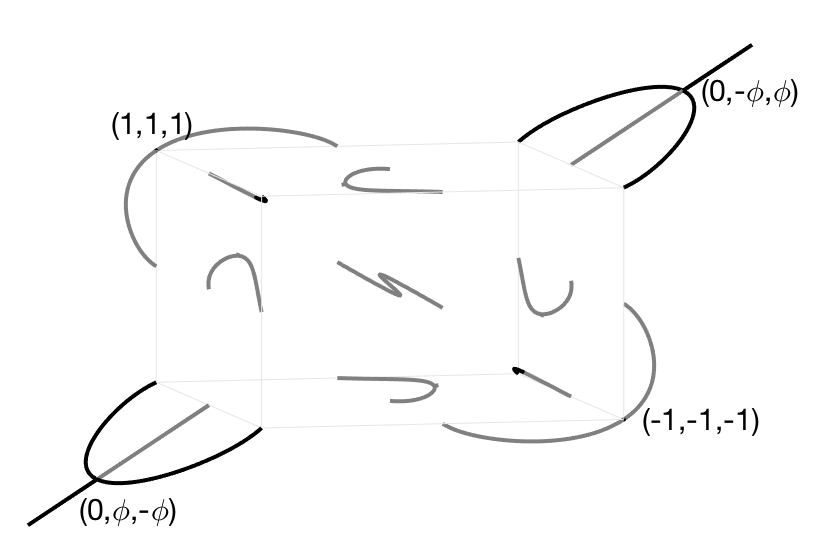}
\caption{This plot is of all the orbits of each fixed point in the domain of the energy function $E_{W,G}$ over varying values of $\kappa$ (the handle of the pitchforks extend out to infinity). The points $\x$ plotted black are stable, while grey marks are the unstable fixed points.  The pitchfork bifurcation happens at $\kappa = \phi / 6$.}
\label{fig:3D}
\end{figure}

We now discuss the nature of the bifurcations in this example. First consider the first bifurcation at $\x_0 = (0,0,0)$ and $\kappa_0 = 1/6$ and set $\y = \x - \x_0$ and $\mu = \kappa - \kappa_0$. We can rewrite our system as
\begin{align*}
\dot{\y} = \frac{1}{3}
\begin{pmatrix}
4 & 1 & -2
\\
1 & 1 & 1
\\
-2 & 1 & 4
\end{pmatrix}
\y
+ 2\mu
\begin{pmatrix}
1 & 1 & -2
\\
1 & -2 & 1
\\
-2 & 1 & 1 
\end{pmatrix}
\y - \y^3.
\end{align*}
The eigenvalues of the left most matrix are $0$, $1$, and $2$. Let $V$ be the matrix whose rows are the corresponding eigenvectors and make the change of variables $\w = V\y$. Then our system becomes
\begin{align*}
\dot{\w} =
\begin{pmatrix}
0 & 0 & 0
\\
0 & 1 & 0
\\
0 & 0 & 2
\end{pmatrix}
\w
+ 6\mu
\begin{pmatrix}
-1 & 0 & 0
\\
0 & 0 & 0
\\
0 & 0 & 1
\end{pmatrix}
\w
- V(V^\top \w)^3.
\end{align*}
From basic bifurcation theory we know that we have a center manifold which is locally represented by two functions $w_2(w_1,\mu)$ and $w_3(w_1,\mu)$ which vanish along with their first partial derivatives at $(w_1,\mu) = (0,0)$. Substituting these functions into the last two equations and using the first shows that
\begin{align*}
w_2(w_1,\mu) = O(\text{cubic terms}) \quad \text{and} \quad w_3(w_1,\mu) = O(\text{cubic terms}) \quad \text{hence} \quad \dot{w_1} = -6 \mu w_1 - \frac{1}{2} w_1^3 + O(\text{sextic terms}).
\end{align*}
This last equation however is locally topologically equivalent to
\begin{align*}
\dot{w_1} = -\mu w_1 - w_1^3
\end{align*}
which represents a pitchfork bifurcation. See \cite{MR2004534}.
\begin{figure}[H]
\center
\includegraphics[width=.4\textwidth]{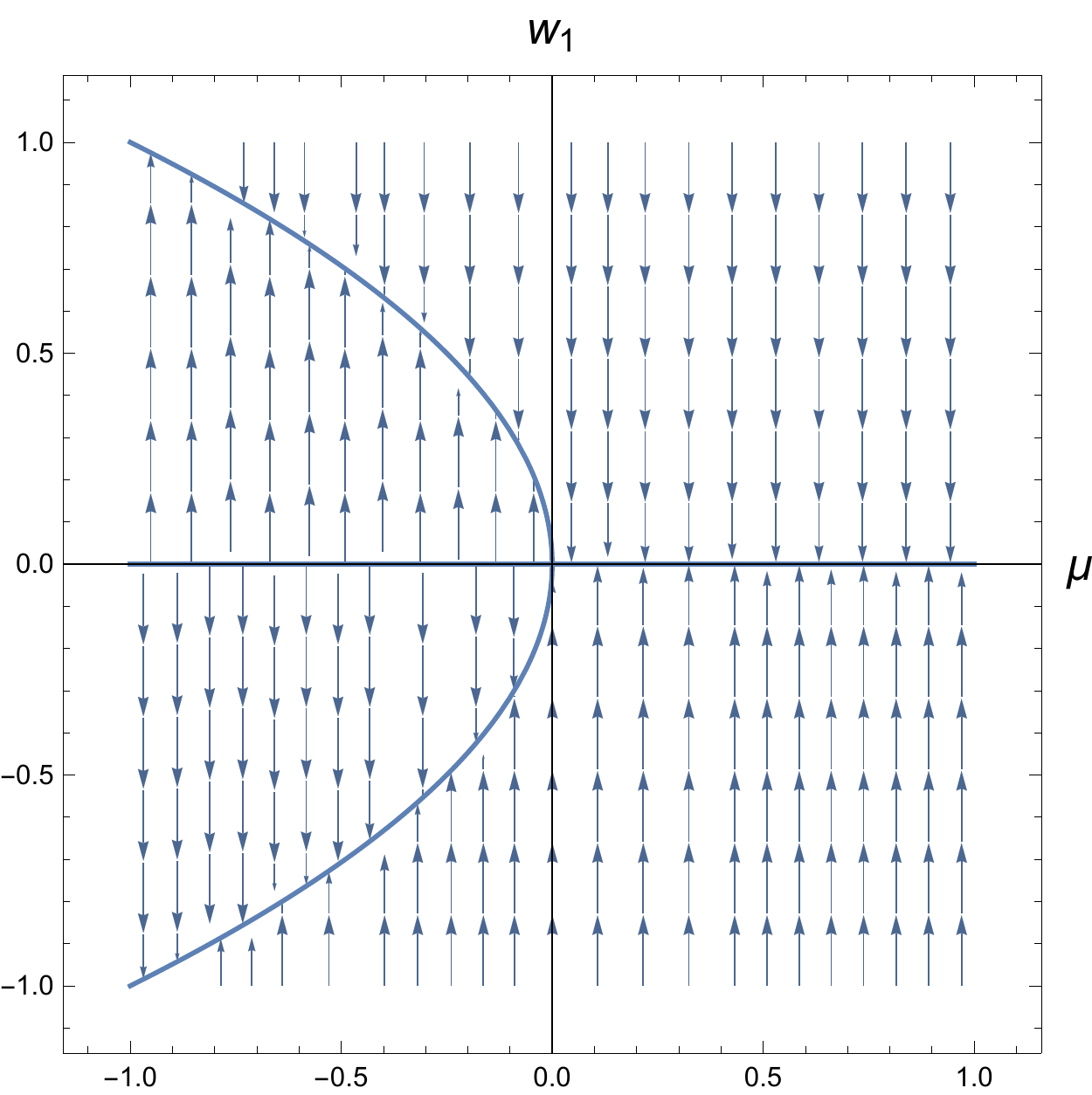}
\caption{The phase portrait of $\dot{w_1} = -\mu w_1 - w_1^3$.}
\end{figure}
Note that this equation only governs the dynamics on the center manifold and that the the system is unstable in the directions corresponding to the other variables $w_2$ and $w_3$ since their corresponding eigenvalues are positive. Therefore when $\mu < 0$, namely, $\kappa < \kappa_0 = 1/6$ we have three fixed points two of which are 2-saddles and one of which is a maximum $w_1 = 0$ which corresponds to $\x = (0,0,0)$. When $\mu > 0$, namely, $\kappa > \kappa_0 = 1/6$ we have only one fixed point, $\x = (0,0,0)$, which is a 2-saddle.

A similar calculation can be used to understand the two other bifurcations. The Jacobian of the bifurcation at $\x = (0,-\phi,\phi)$ and $\kappa = \phi/6$ has eigenvalues $0$, $-3-\sqrt{5}$, and $-3-2\sqrt{5}$. Similar to before this implies that we have to unstable directions and the interesting dynamics occur on a two dimensional center manifold. It is easy to see from Figure $\ref{fig:3D}$ that this is a pitch fork bifurcation and therefore conclude that this bifurcation consists of two minima colliding with a $1-saddle$ to form a single minimum. In the same way the Jacobian for the bifurcation at $\x = (1,1,1)$ and $\kappa = 1/3$ has eigenvalues $0$, $-2$, and $-4$. Thus there are two stable directions. One can show again by a normal form argument that this a pitch fork bifurcation and conclude that it consists of a minima and two 1-saddles colliding to form a single 1-saddle.
\end{example}

\section{Balanced graphs and global minima} \label{sec:global}


\begin{defn} \label{balanced graph}
A graph $G$ is said to be {\bf balanced} if every cycle contains an even number of negative edges.
\end{defn}


One way of describing this type of graph colloquially is the phrase ``the enemy of my enemy is my friend''.  For example, if the graph contains a triangle, then by the definition above, the three vertices in this triangle must all be friends, or exactly two of the pairs must be enemies.  A major result of Cartwright and Harary~\cite{Cartwright.Harary.56, Cartwright.Harary.68} is that all balanced graphs have a type of signed bipartite structure, namely:


\begin{thm}[Cartwright-Harary] \label{thm:CH}
A graph $G$ is balanced if and only if $V$ can be decomposed into two mutually exclusive subsets $V_1$ and $V_2$ such that $\gamma_{ij} \ge 0$ if $i$ and $j$ belong to the same subset and $\gamma_{ij} \le 0$ if $i$ and $j$ belong to different subsets.
\end{thm}


%
%


\begin{thm} \label{global minima}
Suppose that $G$ is a balanced graph, $\kappa > 0$, and $x$ is a global minimum of $E_{G,\kappa}$. Then $x_i \ne 0$ for all $i$ and $x_i$ and $x_j$ have the same sign if and only if $i$ and $j$ belong to the same clique.
\end{thm}



\begin{proof}
Notice that $E(\mathbf{0}) = nW(0)$ and $E(\pm m\mathbf{1})=0$, so that $\mathbf{0}$ is never a global minimum. Thus fix $x \ne \0$, and wlog assume $x_1 > 0$ and $1 \in V_1$.  Define $\widetilde{x}$ by
\[
\widetilde{x}_i = \begin{cases} |x_i| & \mbox{if $i \in V_1$,} \\ -|x_i| & \mbox{if $i \in V_2$.}\end{cases}
\]
We first show that $\gamma_{ij} \widetilde x_i \widetilde x_j \geq 0$ for all $i$ and $j$.  If $i$ and $j$ belong to the same clique, then by Theorem~\ref{thm:CH}, $\gamma_{ij} \ge 0$ and $\widetilde{x}_i$ and $\widetilde{x}_j$ have the same sign.  However, if $i$ and $j$ belong to different cliques, then $\gamma_{ij} \le 0$ and $\widetilde{x}_i$ and $\widetilde{x}_j$ have different signs. 
Since $W$ is even, the transformation $x\mapsto \widetilde x$ does not change the $W$ terms in~\eqref{eq:defofE} and can only make the quadratic terms more negative, so $E_{G,\kappa}(\widetilde{x}) \leq E_{G,\kappa}(x)$.

Now let us suppose that $x$ is a global minimum and $x_i = 0$ for some $i$.  Reusing the argument above gives us sign-definiteness of $\gamma_{ij} \widetilde{x}_j$:  if $i\in V_1$, then $\gamma_{ij}\widetilde{x}_j \ge 0$ for all $j$, and if $i\in V_2$, then $\gamma_{ij}\widetilde{x}_j \le 0$ for all $j$. We compute
\begin{equation*}
\sum_{j \ne i} \gamma_{ij} \widetilde{x}_j = \frac{1}{2\kappa} (W'(\widetilde{x}_i) - (\nabla E(\widetilde x))_i) = 0,
\end{equation*}
but since all $\gamma_{ij}\widetilde{x}_j$ have the same sign, this implies $\gamma_{ij}\widetilde{x}_j = 0$ for all $j$.
Therefore for every $j$ either $\gamma_{ij} = 0$ or $x_j = \widetilde{x}_j = 0$.  This implies that $x_j =0$ for every $j$ that is a neighbor of $i$ in the graph.  Proceeding by induction, this means that $x_j =0$ for any $j$ path-connected to $i$.  Since we assume that $G$ is connected, this implies that $x_j=0$ for all $j$, but we showed above that $0$ is not a global minimum. This is a contradiction, and thus we conclude that $x_i\neq 0$ for all $i$ whenever $x$ is a global minimum.

Finally we show that if $x\neq\widetilde{x}$ for any nonzero $x$, then $E_{G,\kappa}(\widetilde{x}) < E_{G,\kappa}(x)$.  Let us first consider the quantity $\gamma_{ij}(\widetilde{x}_i \widetilde{x}_j - x_i x_j)$. Note that this is either exactly zero, or equal to $2\gamma_{ij}\widetilde{x}_i \widetilde{x}_j > 0$.  (Moreover, this is positive whenever exactly one of the $x_i,x_j$ changes parity when $x\mapsto \widetilde x$.) We then note that
\begin{equation*}
  E_{G,\kappa}(x) - E_{G,\kappa}(\widetilde{x}) = \frac\kappa2 \sum_{i,j=1}^n \gamma_{ij} (\widetilde{x}_i \widetilde{x}_j - x_i x_j),
\end{equation*}
and since each term is nonnegative, as long as any one of these terms are positive, the sum is strictly positive.  Let us now pick $i$ as a vertex where $x_i \neq \widetilde{x}_i$.  Since $G$ is connected, there is a path from $1$ to $i$, i.e. there is a sequence of vertices $n_1,n_2,...,n_k$ such that $n_1 = 1$, $n_k = i$, and $\gamma_{n_j n_{j+1}} \neq 0$ for $\ell=1,\dots,k-1$. 
Since $x_1 = \widetilde{x}_1$ and $x_i \neq \widetilde{x}_i$, there exists a $\ell$ such that $x_{n_\ell} = \widetilde{x}_{n_\ell}$ and $x_{n_{\ell+1}} \neq \widetilde{x}_{n_{\ell+1}}$, which implies
\begin{equation*}
  \gamma_{n_\ell n_{\ell+1}}(\widetilde{x}_{n_\ell} \widetilde{x}_{n_{\ell+1}} - x_{n_\ell} x_{n_{\ell+1}}) = 2\gamma_{n_\ell n_{\ell+1}}\widetilde{x}_{n_\ell} \widetilde{x}_{n_{\ell+1}} > 0,
\end{equation*}
and therefore $E_{G,\kappa}(\widetilde{x}) < E_{G,\kappa}(x)$.

\end{proof}


\section{Non-monotone potentials} \label{sec:nonmonotone}

\subsection{Overview}

It was shown in~\cite{Berglund.Fernandez.Gentz.07.1} that in the case where the underlying graph is a ring, and all of the connections are friendly, that increasing the interaction can only decrease the number of minima, i.e. the number of minima of~\eqref{eq:defofE} is a nonincreasing function of $\kappa$.  In this section, we study pairs $(W,G)$ that do not have this monotonicity property.  We will show that such pairs exist, and, in fact, we can construct pairs that ave arbitrarily more minima for some positive $\kappa$ than for $\kappa=0$.

\begin{defn} \label{def:monotone}
For any pair $(W,G)$ define $E_{G,W,\kappa}(x)$ as above, and let $m_{W,G}(\kappa)$ to be its number of local minima. We say that the pair $(W,G)$ is monotone (resp.~non-monotone) if the function $m_{W,G}$ is monotone non-increasing (resp. ever increases as a function of $\kappa$).  

We also define the quantities
\begin{equation*}
  C(G,W) = \sup_{\kappa \ge 0} \{ m(W,G,\kappa) - m(W,G,0) \}, \quad C(W) = \inf_{G \in \mathcal{G}} C(G,W).
\end{equation*}
Clearly $C(G,W) >0$ iff the pair $G,W$ is nonmonotone, and $C(W)$ is a measure of the ``minimal nonmonotonicity'' that comes from a particular one-dimensional potential.
\end{defn}

\subsection{Main Results}

We have several results that we prove in this section.  The main result, Theorem~\ref{thm:nm}, shows that we can construct a potential $W$ such that the pair $(G,W)$ is always nonmonotone, and, in fact, we can get a uniform bound on how many new minima are created as we increase the coupling.  This theorem has a few corollaries that allow us to bound the size of the coupling in the graph $G$ that would lead to nonmonotonicites.  

\begin{defn}
$\mathcal{W}_s$ denote the subset of $\mathcal{W}$ consisting of potentials that have $s$ inflection points between $x=0$ and $x=m$.  See Figure~\ref{fig:Ws} for an example in $\mathcal{W}_{7}$ after mollification.
\end{defn}


\begin{defn} \label{def:norm}
Define the functional $\norm{\cdot}:\R^{N \times N} \rightarrow \R$ by $\norm{M} = \max_{1 \le i \le n} \sum_{j \ne i} |M_{ij}|$.
\end{defn}

\begin{remark}
It is not hard to show that this quantity has the properties:
\be
\ii $\norm{L(G)} = 0$ if and only if  $\max_{1 \le i \le n} \sum_{j \ne i} |\gamma_{ij}| = 0$
\ii $\| \alpha L(G) \| =  | \alpha | \| L(G) \|$
\ii $\| L^1(G) + L^2(G) \| \le  \| L^1(G) \| + \| L^2(G) \|$.
\ee
Thus $\norm{\cdot}$ is a norm on the space of all zero-row-sum matrices (but is only a seminorm on matrices). In fact it is comparable to the norm $\norm{\cdot}_{1 \rightarrow 1}$.
\end{remark}

\begin{lem}\label{lem:Gershgorin}
  If $\lambda$ is an eigenvalue of $L(G)$, then $\av\lambda \le 2\norm{L(G)}$.
\end{lem}

\begin{proof}
  Gershgorin's theorem states that \\
$$
\lambda \in \bigcup_{i=1}^N \left[\gamma_{ii}-\sum_{j \ne i} |\gamma_{ij}|,\gamma_{ii}+ \sum_{j \ne i} |\gamma_{ij}| \right], 
$$
and since
$$
\bigcup_{i=1}^N \left[\gamma_{ii}-\sum_{j \ne i} |\gamma_{ij}|,\gamma_{ii}+ \sum_{j \ne i} |\gamma_{ij}| \right]  \subset \left[-2 \max_{1 \le i \le n} \sum_{j \ne i} |\gamma_{ij}|, 2 \max_{1 \le i \le n} \sum_{j \ne i} |\gamma_{ij}| \right], 
$$
we have that $\av\lambda \le 2\norm{L(G)}$.
\end{proof}


\begin{lem} \label{lower bound}
Fix $0 < \ell < m < r$ and $M > 0$, and let $W \in \mathcal{W}$ be such that $W'' \ge M$ on $[\ell,r]$. Further let $G \in \mathcal{G}$ and choose $\kappa$ so that
\begin{align*}
\kappa \| L(G) \| = \frac{M \min \{ r-m,m-\ell \}}{2r}.
\end{align*}
Then there exists continuous functions $f_\p : [0,\kappa] \rightarrow ([-r,-\ell] \cup [\ell,r])^n$ for $\p \in \{-m,m\}^n$ satisfying the following conditions:
\begin{enumerate}
\item $f_\p(0) = \p$,
\item $\nabla E(f_\p(\kappa'),\kappa') = 0$ for all $\kappa' \in [0,\kappa]$,
\item $\nabla^2 E(f_\p(\kappa'),\kappa')$ is positive definite for all $\kappa' \in [0,\kappa]$.
\end{enumerate}
\end{lem}



\begin{remark}
Essentially the lemma gives a lower bound on how long the local minima of $E$ at $\kappa = 0$ evolve in $\kappa$ according to the implicit function theorem and therefore how long they exist and remain distinct.
\end{remark}

\begin{proof}
By Lemma~\ref{lem:Gershgorin}, all of the eigenvalues of $L(G)$ are in the range $[-2\norm{L(G)},2\norm{L(G)}]$, and 
by assumption, $W''\ge M$.  From this it follows that $\nabla^2 E(x,\kappa)$ is positive definite on $([-r,-\ell] \cup [\ell,r])^n$ if $0 \le \kappa < \frac{M}{2 \| L(G) \|}$.

Fix $\p \in \{-m,m\}^n$.   Then the implicit function theorem gives us a function $f_\p: [0,\kappa) \rightarrow U_{\p}$ for some $\kappa$, where $U_\p$ is some deleted neighborhood of $\p \in \R^n$.

Now since for any $i$
\begin{align*}
\kappa_{\p} := \frac{-W'(f_\p(\kappa)_i)}{(L(G) f_\p(\kappa))_i} 
= \frac{|W'(f_\p(\kappa)_i)|}{|(L(G) f_\p(\kappa))_i|} 
\ge \frac{\min \{ |W'(\ell)|,|W'(r)| \}}{2r \sum_{j \ne i} |\gamma_{ij}|} 
\ge \frac{M \min \{r-m,m-\ell \}}{2r \| L(G) \|} > 0
\end{align*}
we can extend our implicit function to $f_\p: [0,\kappa_{\p}] \rightarrow ([-r,-\ell] \cup [\ell,r])^n$.
Now since $\p$ was arbitrary in $\{-m,m\}^n$, we obtain the desired result.
\end{proof}



\begin{thm} \label{thm:nm}
For any $s \ge 0$ there exists a potential $W \in \mathcal{W}_{2s+1}$ such that $C(W) \ge 2s$.
\end{thm}


\begin{proof}
We begin with the case $s = 2$. Fix $0 < \ell' < r' < \ell < m < r$ and $M > 0$. Let $E(\kappa, \x) := E_{W,G,\kappa}(\x)$, for any $\kappa$ and $\x$.  Choose $W \in \mathcal{W}_5$ satisfying $W'' \ge M$ on $[\ell',r'] \cup [\ell,r]$ and
\begin{align} \label{ineq: non-monotone potentials}
W'(\ell') & < -\frac{M(r + r') \min \{ r-m,m-\ell \}}{2r}, \quad W'(\ell) < -M \min \{r-m,m-\ell \},
\\
W'(r') & > -\frac{M(\ell - r') \min \{ r-m,m-\ell \}}{2r}, \quad W'(r) > M \min \{r-m,m-\ell \}.
\end{align}
Fix $G \in \mathcal{G}$ and choose $\kappa$ according to Lemma $\ref{lower bound}$. We will show that $E(\kappa, \x)$ 
has a non-zero fixed point $x_0 \notin ([-r,-\ell] \cup [\ell,r])^n$ and therefore conclude that $E(\kappa, \x)$ has at least $2^n+2$ fixed points, namely, each $f_\p(\kappa)$ and $\pm x_0$. This of course implies that $\sup_{\kappa \ge 0} \{ m(W,G,\kappa) - m(W,G,0) \} \ge 2$ which proves the result for $s = 2$.

To find $x_0$, choose a vertex $i$ such that $\| L(G) \| = \sum_{j \ne i} |\gamma_{ij}|$ and let $j$ denote a generic vertex not equal to $i$. Define $R$ to be the rectangular region consisting of all $x \in \R^n$ such that $x_i \in [\ell',r']$, $x_j \in [-r,-\ell]$ if $\gamma_{ij} \ge 0$, and $x_j \in [\ell,r]$ if $\gamma_{ij} < 0$. We will find our $x_0$ in $R$. To do this we first note that
\begin{align*}
\frac{M(\ell - r') \min \{ r-m,m-\ell \}}{2r} & \le \kappa (L(G) x)_i \le \frac{M(r + r') \min \{ r-m,m-\ell \}}{2r},
\\
-M \min \{ r-m,m-\ell \} & \le \kappa(L(G) x)_j \le M \min \{ r-m,m-\ell \}, 
\end{align*}
for all $x \in R$ and therefore conclude that
\begin{align*}
\nabla E(x,\kappa)_i \biggr\rvert_{x_i = \ell'} < 0 < \nabla E(x,\kappa)_i \biggr\rvert_{x_i = r'} \quad \text{and} \quad \nabla E(x,\kappa)_j \biggr\rvert_{x_j = \ell} < 0 < \nabla E(x,\kappa)_j \biggr\rvert_{x_j = r} 
\end{align*}
for all $x$ in the indicated faces of $R$. Therefore by the Poincare--Miranda theorem~\cite{mawhin2013variations}
there exists a critical point $x_0$ in $R$. Finally since we have the same lower bound on $W''$ as in Lemma $\ref{lower bound}$ we see that $\nabla^2 E(x,\kappa)$ is positive definite for all $x \in R$ and therefore conclude that our $x_0$ is in fact a local minimum.

The case $s > 2$ is solved in a similar way by choosing numbers $0 < \ell_1' < r_1' < \dots < \ell_{s-1}' < r_{s-1}' < \ell < m < r$ and imposing the same restrictions to $W'$ as in the $s = 2$ case for each pair $\ell_t'$ and $r_t'$ for $t \in \{1,\dots,s-1\}$ restrictions on $W'$ at these points. All of the resulting fixed points are distinct since the $i$th component lies in the interval $(\ell_t',r_t')$ which is disjoint from the others.
\end{proof}

\subsection{Examples}


\begin{example}
In this example we construct an explicit example of a potential whose existence is guaranteed by Theorem $\ref{thm:nm}$ . We do this for $s = 2$ and therefore we construct a potential $W \in \mathcal{W}_5$ for which $C(W) \ge 4$.  For simplicity choose $\ell_1' = 0 < r_1' = 1 = \ell_2' < r_2'  < \ell = 3 < m = 4 < r = 5$ and $M = 1$. If we define $S(x) = \frac{1}{2} x^2 - \frac{12}{11} x$, then $S(0) = 0$ and $S''(x) = 1 = M$ and
\begin{gather*}
S'(0) = -\frac{12}{11} < -\frac{2}{3} \le -\frac{5+r'}{10} = -\frac{M(r+r') \min \{r-m,m-\ell \}}{2r},
\\
S'(1) = -\frac{1}{11} > -\frac{1}{10} \ge - \frac{3 - r'}{10} = -\frac{M(\ell-r') \min \{r-m,m-\ell \}}{2r}, 
\end{gather*}
for $r' \in \{1,2\}$. If we further define $T(x) = x^2 - 4$, then $T(-2) = 0$ and $T''(x) = 2 \ge 1 = M$ and
\begin{align*}
T'(-1) = -2 < -1 = -M \min \{r-m,m-\ell \}, \quad T'(1) = 2 > 1 = M \min \{r-m,m-\ell \}.
\end{align*}
Thus we can define our potential for non-negative arguments by
\begin{align} \label{W_7 example}
W(x) =
\begin{cases}
S(x) & \mbox{if $0 \le x < 1$,}
\\
S(1) + S(x-1) & \mbox{if $1 \le x < 2$,}
\\
2S(1) + B(x-4) & \mbox{if $x \ge 2$,}
\end{cases}
\end{align}
and extend to negative values by symmetry.  For simplicity of definition, this potential is not smooth;  it is not hard to see that we could mollify the corners without changing any of the arguments below.

By construction the intervals $(-2,-1)$, $(-1,0)$, $(0,1)$, and $(1,2)$ each contribute at least one new fixed point at $\kappa = \frac{1}{10 \| L(G) \|}$ which results in an increase of four fixed points as desired. We give a plot of our potential in the figure below. Furthermore let us choose $G$ to be the complete graph on three vertices with unit edge weights. Then $\|L(G)\| = 1$ and we are guaranteed to have four new fixed points at $\kappa = \frac{1}{10}$. We give a table of the number of minima of $E$ for our potential and graph at certain values of $\kappa$ below.

\begin{figure}
\centering
\begin{subfigure}{.5\textwidth}
  \centering
  \includegraphics[width=1\linewidth]{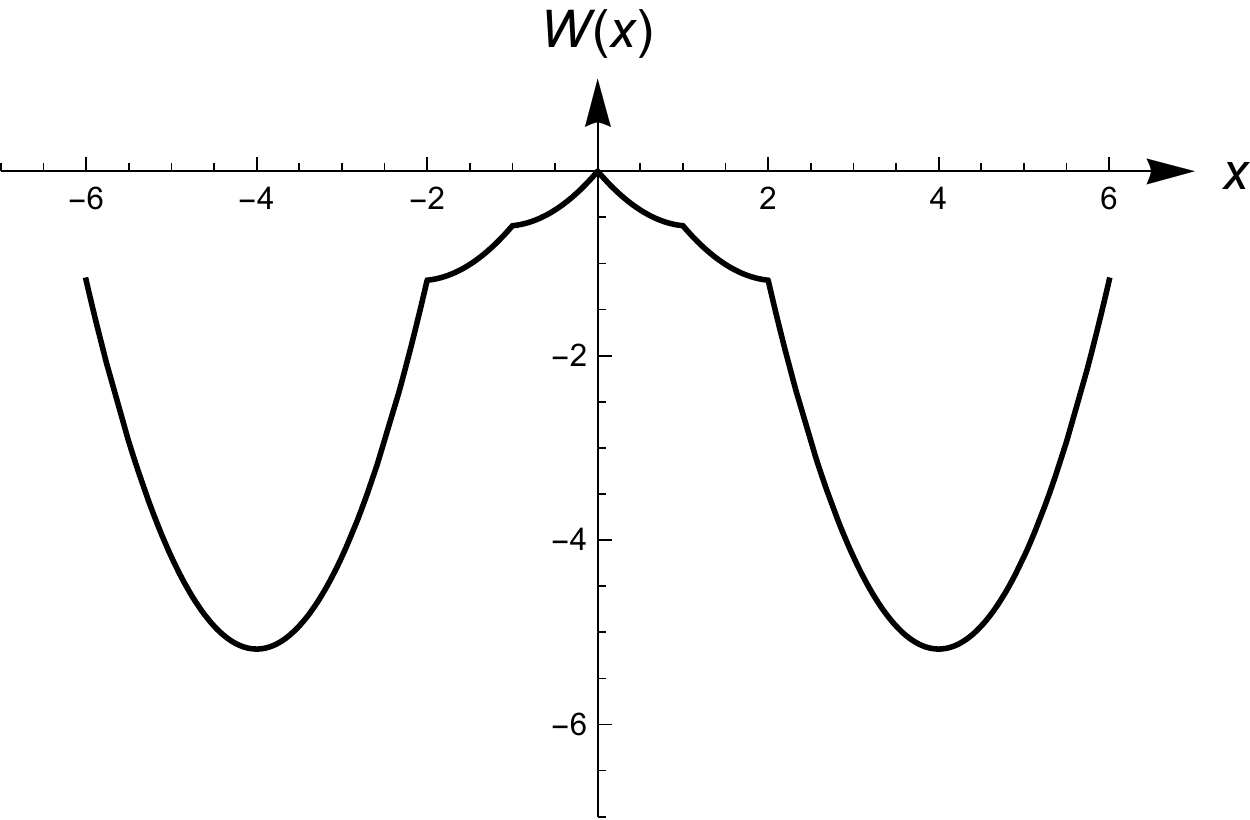}
\end{subfigure}%
\begin{subfigure}{.5\textwidth}
\centering
  \begin{tabular}{| c | c |}
\hline
$\kappa$ & $m_{W,G}(\kappa)$ \\
\hline
0 & 8 \\
\hline
0.02 & 63 \\
\hline
0.04 & 98 \\
\hline
0.06 & 98 \\
\hline
0.08 & 97 \\
\hline
0.1 & 90 \\
\hline
0.12 & 77 \\
\hline
0.14 & 65 \\
\hline
0.16 & 65 \\
\hline
0.18 & 56 \\
\hline
0.2 & 50 \\
\hline
\end{tabular}
\end{subfigure}
\caption{A plot of the potential $W$ defined in~\ref{W_7 example}, and a table of the function $m_{W,G}$ for the potential $W$ and graph on three vertices with unit edge weights edge weights. The function $m_{W,G}$ appears to have a local maximum around 0.05.}
\label{fig:Ws}
\end{figure}

We see that we obtain over one hundred more new minima. Also we notice that $m(W,G,\kappa)$ appears to peak before we reach the value of $\kappa$ used in the proof of Theorem $\ref{thm:nm}$.

We obtained numerical estimates for the number of minima at each value of $\kappa$ by using a Monte Carlo method evolving each point under our gradient flow until we determine that we are sufficiently close to a minimum. Note that our number are therefore lower bounds for the actual values of $m(W,G,\kappa)$.
\end{example}

\begin{example}
In this example we show that we may actually achieve way more minima than predicted by our theorem. Let $\kappa = 1$, $\epsilon = .01$, and $W$ be the classical W-potential on the set 
$(-\infty,-1) \cup (-1/4+2\epsilon, 1/4-2\epsilon) \cup (1,\infty)$, decreasing and smooth on $(1/4-2\epsilon,1)$, increasing and smooth on $(-1,-1/4+2\epsilon)$, and even.  Additionally, suppose that 
\begin{align*}
W'(x) =
\begin{cases}
\epsilon	&	\mbox{for $x = -1/4,-1/2,$ and $-3/4,$}\\
10		&	\mbox{for $x = -1/4+\epsilon,-1/2+\epsilon,-3/4+\epsilon,$ and $-1+\epsilon$,}\\
-10		&	\mbox{for $x = 1/4-\epsilon,1/2-\epsilon,3/4-\epsilon,$ and $1-\epsilon$,}\\
-\epsilon	&	\mbox{for $x = 1/4,1/2,$ and $3/4,$.}\\
\end{cases}
\end{align*}
Now consider the graph with two nodes and edge of weight $1$. Let $L$ be the corresponding graph Laplacian. Then the gradient of the energy is 
\begin{equation*}
\nabla E_{W,G,\kappa}(x) = \begin{pmatrix} W'(x_1)+x_1-x_2 \\ W'(x_2)-x_1+x_2 \end{pmatrix}
\end{equation*}
Let $f_1(x_1,x_2) = W'(x_1)+x_1-x_2$ and $f_2(x_1,x_2) = W'(x_2)-x_1+x_2$.  Consider the sets $A_n = [-n/4, -n/4+\epsilon]$ and $B_n = [n/4+\epsilon, n/4]$ for $n=1,2,3,4$.  There are 16 cartesian products of the form $A_i \times B_j$.

Consider $(x_1,x_2) \in A_i \times B_j$ for some $1 \leq i,j \leq 4$.  Then certainly $x_1-x_2 \in [-2,-2/4+2\epsilon]$ and $-x_1+x_2 \in [2/4-2\epsilon,2]$.  So we have that 
\begin{align*}
f_1(x_1,x_2) = 
\begin{cases}
W'(x_1)+x_1-x_2 \leq -2/4+3\epsilon<0	&	\mbox{for $x_1 = -1/4,-1/2,-3/4,$ and $-1$,}\\
W'(x_1)+x_1-x_2 \geq 10-2 > 0	&	\mbox{for $x_1 = -1/4+\epsilon,-1/2+\epsilon,-3/4+\epsilon,$ and $-1+\epsilon$,}\\
\end{cases}
\end{align*}
and
\begin{align*}
f_2(x_1,x_2) = 
\begin{cases}
W'(x_2)-x_1+x_2 \geq 2/4-3\epsilon>0	&	\mbox{for $x_2 = 1/4,1/2,3/4,$ and $1$,}\\
W'(x_2)-x_1+x_2 \leq -10+2 < 0	&	\mbox{for $x_2 = 1/4-\epsilon,1/2-\epsilon,3/4-\epsilon,$ and $-1+\epsilon$.}\\
\end{cases}
\end{align*}
Applying the Poincare--Miranda theorem we get 16 local extrema, one for each of the 16 sets, of the form $x_1 < 0 < x_2$, by symmetry there are another 16 of the form $x_2 < 0 < x_1$.  It is left to the reader that these extrema happen in regions with a positive Hessian.  There are an additional 2 local minima at $(1,1)$ and $(-1,-1)$, giving a total number of 34 minima at $\kappa = 1$. \end{example}

\begin{example}
The main drivers in the increase in minima in the proofs above is due to the inflection points in the potential, which gave rise to ``shelves'' that would separate out the different points.  In this example, we show that these shelves are useful but not necessary for nonmonotonicity.

Let $\kappa = 0.1$. Let $W$ be a smooth, even function with only two inflection points so that $W(x) = (|x| - 1)^4$ on $(-\infty,-1/2] \cup [1/2,\infty)$.   Now consider the graph with two nodes and edge of weight $1$. Let $L$ be the corresponding graph Laplacian. Then the gradient of the energy 
\begin{equation*}
\nabla E_{W,G,\kappa}(x) = \begin{pmatrix} W'(x_1)-.1x_1+.1x_2 \\ W'(x_2)+.1x_1-.1x_2 \end{pmatrix}
\end{equation*}
Let $f_1(x_1,x_2) = W'(x_1)-.1x_1+.1x_2$ and $f_2(x_1,x_2) = W'(x_2)+.1x_1-.1x_2$.  We have that
\begin{align*}
f_1(x_1,x_2) = 
\begin{cases}
W'(x_1)-.1x_1+.1x_2 < 0-.1-.1 < 0 	&	\mbox{for $x_1 = 1, x_2 \in (-2,-1)$}\\
W'(x_1)-.1x_1+.1x_2 > 4-.2-.2 > 0	&	\mbox{for $x_1 = 2, x_2 \in (-2,-1)$}\\
\end{cases}
\end{align*}
and
\begin{align*}
f_2(x_1,x_2) = 
\begin{cases}
W'(x_2)+.1x_1-.1x_2 < -4+.2+.2 < 0	&	\mbox{for $x_2 = -2, x_1 \in (1,2)$}\\
W'(x_2)+.1x_1-.1x_2 > 0+.1+.1 > 0	&	\mbox{for $x_2 = -1, x_1 \in (1,2)$.}\\
\end{cases}
\end{align*}

By the Poincare--Miranda theorem we have an extrema in the square $(-2,-1) \times (1,2)$, and by symmetry a second extrema.  Also we have

\begin{align*}
f_1(x_1,x_2) = 
\begin{cases}
W'(x_1)-.1x_1+.1x_2 < -1/2-.05-.1 < 0 	&	\mbox{for $x_1 = 1/2, x_2 \in (1,2)$}\\
W'(x_1)-.1x_1+.1x_2 > 0-.1+.1 > 0		&	\mbox{for $x_1 = 1, x_2 \in (1,2)$}\\
\end{cases}
\end{align*}
and
\begin{align*}
f_2(x_1,x_2) = 
\begin{cases}
W'(x_2)+.1x_1-.1x_2 < 0+.1-.1 = 0	&	\mbox{for $x_2 = 1, x_1 \in (1/2,1)$}\\
W'(x_2)+.1x_1-.1x_2 > 4+.05-.1 > 0	&	\mbox{for $x_2 = 2, x_1 \in (1/2,1)$.}\\
\end{cases}
\end{align*}
By the Poincare--Miranda theorem we have a extrema in the square $(1/2,1) \times (1,2)$.  By symmetry there are three more extrema in $(1,2) \times (1/2,1)$, $(-1,-1/2) \times (-2,-1)$, and $(-2,-1) \times (-1,-1/2)$.  It is left to the reader that these extrema happen in regions with a positive Hessian.  Thus we have a total of 6 minima at $\kappa = 0.1$.  Based on this example, one may conjecture that for any W-potential there exists a graph with negative edge weights so that minima increase locally.  
\end{example}


\begin{example}
This example gives an increase in minima with only 2 inflection points like the last example.  
This example will not require negative edge weights.
Let $\kappa = 3/40$ and $\epsilon = 0.01$.  
Let $W$ be a smooth, even function with only two inflection points so that $W$ is the classical potential on $(-\infty,-1] \cup [1,\infty)$, is smooth, and has second derivative $W''(x) > 0$ on $(-1,-1+\epsilon) \cup (1-\epsilon,1)$.  Also suppose the following about $W'$,

\begin{align*}
W'(x) =
\begin{cases}
10		&	\mbox{for $x \in [-1 + \epsilon,-1/2 - \epsilon]$,}\\
1/10		&	\mbox{for $x \in [-1/2 + \epsilon,-\epsilon]$,}\\
-1/10		&	\mbox{for $x \in [\epsilon,1/2 - \epsilon]$,}\\
-10		&	\mbox{for $x \in [1/2 + \epsilon,1 - \epsilon]$,}\\
\end{cases}
\end{align*}

Now consider the graph with two nodes and edge of weight $1$. Let $L$ be the corresponding graph Laplacian. Then the gradient of the energy 
\begin{equation*}
\nabla E_{W,G,\kappa}(x) = \begin{pmatrix} W'(x_1)+\frac{3}{40}x_1-\frac{3}{40}x_2 \\ W'(x_2)-\frac{3}{40}x_1+\frac{3}{40}x_2 \end{pmatrix}
\end{equation*}
Let $f_1(x_1,x_2) = W'(x_1)+\frac{3}{40}x_1-\frac{3}{40}x_2$ and $f_2(x_1,x_2) = W'(x_2)-\frac{3}{40}x_1+\frac{3}{40}x_2$.  We have that 

\begin{align*}
f_1(x_1,x_2) = 
\begin{cases}
W'(x_1)+\frac{3}{40}x_1-\frac{3}{40}x_2 < 0+\frac{3}{40}(-1-\epsilon) < 0 	&	\mbox{for $x_1 = -1, x_2 \in (\epsilon,1/2-\epsilon)$}\\
W'(x_1)+\frac{3}{40}x_1-\frac{3}{40}x_2 > 10+\frac{3}{40}(-1+\epsilon-(1/2-\epsilon)) > 0	&	\mbox{for $x_1 = -1+\epsilon, x_2 \in (\epsilon,1/2-\epsilon)$}\\
\end{cases}
\end{align*}
and
\begin{align*}
f_2(x_1,x_2) = 
\begin{cases}
W'(x_2)-\frac{3}{40}x_1+\frac{3}{40}x_2 < -1/10+\frac{3}{40}(-(-1)+\epsilon) < 0	&	\mbox{for $x_2 = \epsilon, x_1 \in (-1,-1+\epsilon)$}\\
W'(x_2)-\frac{3}{40}x_1+\frac{3}{40}x_2 > -1/10+\frac{3}{40}(-(-1+\epsilon)+1/2-\epsilon) > 0	&	\mbox{for $x_2 = 1/2-\epsilon, x_1 \in (-1,-1+\epsilon)$}\\
\end{cases}
\end{align*}

By the Poincare--Miranda theorem we have an extrema in the square $(-1,-1+\epsilon) \times (\epsilon,1/2-\epsilon)$, and by symmetry a second extrema.  Now for any point in the open set $(-1,-1+\epsilon) \times (\epsilon,1/2-\epsilon)$ there exists a $\delta > 0$ so that
\begin{align*}
D^2E \geq 
\begin{pmatrix}
\delta && 0 \\
0 && 0 \\
\end{pmatrix}
+
\begin{pmatrix}
1 && -1 \\
-1 && 1 \\
\end{pmatrix}
=
\begin{pmatrix}
1 + \delta && -1 \\
-1 && 1 \\
\end{pmatrix} 
\end{align*}
which is positive.

Now we also have that
\begin{align*}
f_1(x_1,x_2) = 
\begin{cases}
W'(x_1)+\frac{3}{40}x_1-\frac{3}{40}x_2 < 0+\frac{3}{40}(-1-\epsilon) < 0 	&	\mbox{for $x_1 = -1, x_2 \in (1-\epsilon,1)$}\\
W'(x_1)+\frac{3}{40}x_1-\frac{3}{40}x_2 > 10+\frac{3}{40}(-1+\epsilon-(1/2-\epsilon)) > 0	&	\mbox{for $x_1 = -1+\epsilon, x_2 \in (1-\epsilon,1)$}.\\
\end{cases}
\end{align*}
By the Poincare--Miranda theorem we have a extrema in the square $(-1,1+\epsilon) \times (1-\epsilon,1)$.  By symmetry there is another extrema.  Here the Hessian is easier to bound so it is left to the reader.  There are 2 more minima at $(1,1)$ and $(-1,-1)$.  Thus there are at least 6 minima at $\kappa = 3/40$.  This example shows that minima may increase locally on some W-potentials with graphs of all positive edge weights.
\end{example}


\section{Conclusions}\label{sec:outtro}

There were two directions explored in this paper:  In Section~\ref{sec:global} we considered the global mimima of the energy functional whenever the graph topology is balanced, and in Section~\ref{sec:nonmonotone} we studied the nonmonotonicity of the number of minima.  

In Section~\ref{sec:global} we only considered a particular class of graphs that had a natural structure that led to our being able to describe the global minimum.  Note also that the configuration which globally minimized the energy was independent of the coupling strength $\kappa$ (although of course its energy changes as $\kappa$ changes).  We conjecture that this property is held (with some trivial exceptions) only by balanced graphs, i.e. if a graph is not balanced, then the minimum-energy configuration changes as a function of $\kappa$.  

The results of Section~\ref{sec:nonmonotone} are a bit technical, but they show a surprising fact, if we consider the thermalization of such potentials.  For example, we could add small white noise to any of these ODEs, and we know that all of the (local) minima identified above now become metastable.  An  observer who could only observe the nonequilibrium behavior of our potentials would not be able to detect the ``shelves'' in the one-dimensional potentials, but these shelves play a huge role when these potentials are coupled together.


\bibliography{social}

\newcommand{\etalchar}[1]{$^{#1}$}
\providecommand{\bysame}{\leavevmode\hbox to3em{\hrulefill}\thinspace}
\providecommand{\MR}{\relax\ifhmode\unskip\space\fi MR }
\providecommand{\MRhref}[2]{%
  \href{http://www.ams.org/mathscinet-getitem?mr=#1}{#2}
}
\providecommand{\href}[2]{#2}
\begin{thebibliography}{YAO{\etalchar{+}}11}

\bibitem[ACTA17]{Ashwin.etal.17}
Peter Ashwin, Jennifer Creaser, and Krasimira Tsaneva-Atanasova, \emph{Fast and
  slow domino effects in transient network dynamics}, arXiv preprint
  arXiv:1701.06148 (2017).

\bibitem[AJGA14]{Auoay.etal.14}
Saoussen Aouay, Salma Jamoussi, Faiez Gargouri, and Ajith Abraham,
  \emph{Modeling dynamics of social networks: A survey}, Computational Aspects
  of Social Networks (CASoN), 2014 6th International Conference on, IEEE, 2014,
  pp.~49--54.

\bibitem[Alt12]{Altafini.12}
Claudio Altafini, \emph{Dynamics of opinion forming in structurally balanced
  social networks}, PloS one \textbf{7} (2012), no.~6, e38135.

\bibitem[BFG07a]{Berglund.Fernandez.Gentz.07.1}
Nils Berglund, Bastien Fernandez, and Barbara Gentz, \emph{Metastability in
  interacting nonlinear stochastic differential equations: I. from weak
  coupling to synchronization}, Nonlinearity \textbf{20} (2007), no.~11, 2551.

\bibitem[BFG07b]{Berglund.Fernandez.Gentz.07.2}
\bysame, \emph{Metastability in interacting nonlinear stochastic differential
  equations: Ii. large-n behaviour}, Nonlinearity \textbf{20} (2007), no.~11,
  2583.

\bibitem[BRG16]{Berghardt.Rand.Girvan.16}
Keith Burghardt, William Rand, and Michelle Girvan, \emph{Competing opinions
  and stubborness: connecting models to data}, Physical Review E \textbf{93}
  (2016), no.~3, 032305.

\bibitem[CFL09]{Castellano.Fortunato.Loreto.09}
Claudio Castellano, Santo Fortunato, and Vittorio Loreto, \emph{Statistical
  physics of social dynamics}, Reviews of modern physics \textbf{81} (2009),
  no.~2, 591.

\bibitem[CH56]{Cartwright.Harary.56}
Dorwin Cartwright and Frank Harary, \emph{Structural balance: a generalization
  of heider's theory.}, Psychological review \textbf{63} (1956), no.~5, 277.

\bibitem[CH68]{Cartwright.Harary.68}
D.~Cartwright and F.~Harary, \emph{On the coloring of signed graphs}, Elem.
  Math. \textbf{23} (1968), 85--89. \MR{0233732}

\bibitem[DGM14]{Das.Gollapudi.Mungala.14}
Abhimanyu Das, Sreenivas Gollapudi, and Kamesh Munagala, \emph{Modeling opinion
  dynamics in social networks}, Proceedings of the 7th ACM international
  conference on Web search and data mining, ACM, 2014, pp.~403--412.

\bibitem[DL94]{Durrett.Levin.94}
Richard Durrett and Simon~A Levin, \emph{Stochastic spatial models: a user's
  guide to ecological applications}, Philosophical Transactions of the Royal
  Society of London B: Biological Sciences \textbf{343} (1994), no.~1305,
  329--350.

\bibitem[DPLM14]{DaiPra.Louis.Minelli.14}
Paolo Dai~Pra, Pierre-Yves Louis, and Ida~G Minelli, \emph{Synchronization via
  interacting reinforcement}, Journal of Applied Probability \textbf{51}
  (2014), no.~02, 556--568.

\bibitem[GS13]{Ghaderi.Srikant.13}
Javad Ghaderi and R~Srikant, \emph{Opinion dynamics in social networks: A local
  interaction game with stubborn agents}, American Control Conference (ACC),
  2013, IEEE, 2013, pp.~1982--1987.

\bibitem[HL75]{Holley.Liggett.75}
Richard~A. Holley and Thomas~M. Liggett, \emph{Ergodic theorems for weakly
  interacting infinite systems and the voter model}, Ann. Probability
  \textbf{3} (1975), no.~4, 643--663. \MR{0402985}

\bibitem[HL78]{Holley.Liggett.78}
R.~Holley and T.~M. Liggett, \emph{The survival of contact processes}, Ann.
  Probability \textbf{6} (1978), no.~2, 198--206. \MR{0488379}

\bibitem[JGN01]{Jin.Girvan.Newman.01}
Emily~M Jin, Michelle Girvan, and Mark~EJ Newman, \emph{Structure of growing
  social networks}, Physical review E \textbf{64} (2001), no.~4, 046132.

\bibitem[KLB09]{Kunegis.Lommatzsch.Bauchhage.09}
J{\'e}r{\^o}me Kunegis, Andreas Lommatzsch, and Christian Bauckhage, \emph{The
  slashdot zoo: mining a social network with negative edges}, Proceedings of
  the 18th international conference on World wide web, ACM, 2009, pp.~741--750.

\bibitem[KSL{\etalchar{+}}10]{Kunegis.etal.12}
J{\'e}r{\^o}me Kunegis, Stephan Schmidt, Andreas Lommatzsch, J{\"u}rgen Lerner,
  Ernesto~W De~Luca, and Sahin Albayrak, \emph{Spectral analysis of signed
  graphs for clustering, prediction and visualization}, Proceedings of the 2010
  SIAM International Conference on Data Mining, SIAM, 2010, pp.~559--570.

\bibitem[Lig99]{Liggett.book}
Thomas~M. Liggett, \emph{Stochastic interacting systems: contact, voter and
  exclusion processes}, Grundlehren der Mathematischen Wissenschaften
  [Fundamental Principles of Mathematical Sciences], vol. 324, Springer-Verlag,
  Berlin, 1999. \MR{1717346}

\bibitem[Maw13]{mawhin2013variations}
Jean Mawhin, \emph{Variations on poincar{\'e}-miranda's theorem}, Advanced
  Nonlinear Studies \textbf{13} (2013), no.~1, 209--217.

\bibitem[NBW11]{Newman.Barabasi.Watts.book}
Mark Newman, Albert-Laszlo Barabasi, and Duncan~J Watts, \emph{The structure
  and dynamics of networks}, Princeton University Press, 2011.

\bibitem[NG04]{Newman.Girvan.04}
Mark~EJ Newman and Michelle Girvan, \emph{Finding and evaluating community
  structure in networks}, Physical review E \textbf{69} (2004), no.~2, 026113.

\bibitem[SAR08]{Sood.Antal.Redner.08}
V.~Sood, Tibor Antal, and S.~Redner, \emph{Voter models on heterogeneous
  networks}, Phys. Rev. E (3) \textbf{77} (2008), no.~4, 041121, 13.
  \MR{2495459}

\bibitem[Str01]{Strogatz.01}
Steven~H Strogatz, \emph{Exploring complex networks}, Nature \textbf{410}
  (2001), no.~6825, 268--276.

\bibitem[SWS00]{Sznajd.00}
Katarzyna Sznajd-Weron and Jozef Sznajd, \emph{Opinion evolution in closed
  community}, International Journal of Modern Physics C \textbf{11} (2000),
  no.~06, 1157--1165.

\bibitem[Wig03]{MR2004534}
Stephen Wiggins, \emph{Introduction to applied nonlinear dynamical systems and
  chaos}, second ed., Texts in Applied Mathematics, vol.~2, Springer-Verlag,
  New York, 2003. \MR{2004534}

\bibitem[WS98]{Watts.Strogatz.98}
Duncan~J Watts and Steven~H Strogatz, \emph{Collective dynamics of
  ‘small-world’networks}, nature \textbf{393} (1998), no.~6684, 440--442.

\bibitem[YAO{\etalchar{+}}11]{Yildiz.etal.11}
Ercan Yildiz, Daron Acemoglu, Asuman~E Ozdaglar, Amin Saberi, and Anna
  Scaglione, \emph{Discrete opinion dynamics with stubborn agents}.

\end{thebibliography}
\bibliographystyle{amsalpha}


\end{document}